\documentclass[11pt]{amsart}
\usepackage[margin=30mm]{geometry}
\usepackage{amsmath,amssymb}
\usepackage{amsthm}
\usepackage{mathrsfs}

\newtheorem{thm}{Theorem}[section]

\newtheorem{lem}{Lemma}[section]
\newtheorem{cor}{Corollary}[section]
\newtheorem{defi}{Definition}[section]

\newtheorem{rem}{Remark}[section]

\newtheorem*{claim}{\it{Claim}}
\newtheorem*{theorem}{\it{Theorem}}

\usepackage{enumitem}

\begin{document}

\title{Sufficient conditions for distributional chaos of type I}
\author{Noriaki Kawaguchi}
\address{Department of Mathematical and Computing Science, School of Computing, Institute of Science Tokyo, 2-12-1 Ookayama, Meguro-ku, Tokyo 152-8552, Japan}
\email{gknoriaki@gmail.com}

\begin{abstract}
Distributional chaos of type I (DC1) is a stronger variant of Li--Yorke chaos. In this paper, we consider the fact that the time-one map of a mixing Anosov flow exhibits DC1 and generalize it to obtain simple sufficient conditions for DC1.
\end{abstract}

\maketitle

\markboth{NORIAKI KAWAGUCHI}{Sufficient conditions for distributional chaos of type I}

\section{Introduction}

The concept of {\em chaos} plays a central role in the modern theory of dynamical systems. A mathematical definition of chaos, so-called  {\em Li--Yorke chaos}, was given in \cite{LY}. In \cite{SS}, three statistical variants of Li--Yorke chaos, collectively called {\em distributional chaos}, were introduced (initially) for interval maps. Among the three, {\em distributional chaos of type I} (DC1) is the strongest one, representing a stronger variant of Li--Yorke chaos. While the Li--Yorke type chaos has been generalized in terms of Furstenberg families, this paper focuses specifically on DC1 (see \cite{LYe, O} for background). We show that the time-one map of a mixing Anosov flow exhibits DC1 and generalize it to obtain sufficient conditions for DC1.
 
We begin by defining DC1 which is generalized for $n$-tuples, $n\ge2$, in \cite{LO, TF}. Throughout, $X$ denotes a compact metric space endowed with a metric $d$. 

\begin{defi}
\normalfont
Given a continuous map $f\colon X\to X$, an $n$-tuple $(x_1,x_2,\dots,x_n)\in X^n$, $n\ge2$, is said to be {\em DC1-$n$-$\delta$-scrambled} for $\delta>0$ if
\begin{equation*}
\limsup_{m\to\infty}\frac{1}{m}|\{0\le i\le m-1\colon\min_{1\le j<k\le n}d(f^i(x_j),f^i(x_k))>\delta\}|=1,
\end{equation*}
and
\begin{equation*}
\limsup_{m\to\infty}\frac{1}{m}|\{0\le i\le m-1\colon\max_{1\le j<k\le n}d(f^i(x_j),f^i(x_k))<\epsilon\}|=1
\end{equation*}
for all $\epsilon>0$. Let ${\rm DC1}_n^\delta(X,f)$ denote the set of DC1-$n$-$\delta$-scrambled $n$-tuples and let
\[
{\rm DC1}_n(X,f)=\bigcup_{\delta>0}{\rm DC1}_n^\delta(X,f).
\]
A subset $S$ of  $X$ is said to be {\em DC1-$n$-scrambled} (resp.\:{\em DC1-$n$-$\delta$-scrambled}) if
\[
(x_1,x_2,\dots ,x_n)\in{\rm DC1}_n(X,f)\:\text{(resp.\:${\rm DC1}_n^\delta(X,f)$)}
\]
for all distinct $x_1,x_2,\dots,x_n\in S$. We say that $f$ exhibits the {\em distributional $n$-chaos of type I} (${\rm DC1}_n$) if there is an uncountable DC1-$n$-scrambled subset of $X$.
\end{defi}

We recall a simplified version of a theorem of Mycielski \cite{M}. Mycielski's theorem is used in \cite{BGKM} to prove that positive topological entropy implies Li--Yorke chaos. A comprehensive treatment of Mycielski's theorem and some of its applications to topological dynamics are given in \cite{A}.

For a topological space $Z$, a subset $S$ of $Z$ is said to be {\em residual} if $S$ contains a countable intersection of dense open subsets of $Z$. By Baire category theorem, if $Z$ is a complete metric space, then every countable intersection of residual subsets of $Z$ is dense in $Z$. A topological space $Z$ is said to be {\em perfect} if $Z$ has no isolated point. For a complete metric space $Z$, a subset $S$ of $Z$ is said to be a {\em Mycielski set} if $S$ is a union of countably many Cantor sets (see \cite{BGKM}). Note that for a Mycielski set $S$ in $Z$ and an open subset $U$ of $Z$ with $S\cap U\ne\emptyset$, $S\cap U$ is an uncountable set.  

\begin{theorem}[Mycielski]
Let $Z$ be a perfect complete separable metric space. If $R_n$ is a residual subset of $Z^n$ for each $n\ge2$, then there is a Mycielski set $S$ which is dense in $Z$ and satisfies $(x_1,x_2,\dots,x_n)\in R_n$ for all $n\ge2$ and distinct $x_1,x_2,\dots,x_n\in S$.
\end{theorem}

\begin{rem}
\normalfont
Let $f\colon X\to X$ be a continuous map. By the above theorem, for a sequence $(\delta_n)_{2\le n< b+1}$ of positive numbers, where $2\le b\le\infty$, if
\[
{\rm DC1}_n^{\delta_n}(X,f)
\]
is a residual subset of $X^n$ for each $2\le n<b+1$, then there is a dense Mycielski set $S$ in $X$ such that $S$ is DC1-$n$-$\delta_n$-scrambled for all $2\le n<b+1$, in particular, $f$ exhibits ${\rm DC1}_n$ for all $2\le n<b+1$.
\end{rem}

Let $M$ be a closed differentiable manifold endowed with a Riemannian metric $d$ and  let $F\colon\mathbb{R}\times M\to M$ be a mixing Anosov flow. {\em  Anosov} means that $M$ is a hyperbolic set for $F$ (see \cite{FH} for details and background information). Let $F^t(x)=F(t,x)$ for all $(t,x)\in\mathbb{R}\times M$ and let $g=F^1\colon M\to M$, the time-one map for $F$. Let
\[
W^{ss}(x)=\{y\in M\colon\lim_{t\to\infty}d(F^t(x),F^t(y))=0\},
\]
$x\in M$, and note that
\[
W^{ss}(x)=\{y\in M\colon\lim_{i\to\infty}d(g^i(x),g^i(y))=0\}
\]
for all $x\in M$. By results in Chapter 6 of \cite{FH}, a periodic point $p$ for $F$ satisfies $M=\overline{W^{ss}(p)}$. Let $\tau$ denote the period of $p$, i.e.,
\[
\tau=\min\{t>0\colon F^t(p)=p\}>0,
\]
and let $\Lambda$ denote the orbit of $p$:
\[
\Lambda=\{F^t(p)\colon t\in\mathbb{R}\}=\{F^t(p)\colon t\in[0,\tau)\}.
\]
We have
\[
\overline{W^{ss}(F^t(p))}=\overline{F^t(W^{ss}(p))}=F^t(\overline{W^{ss}(p)})=F^t(M)=M
\]
for all $t\in\mathbb{R}$; therefore, $M=\overline{W^{ss}(x)}$ for all $x\in\Lambda$. Let $S^1=\{z\in\mathbb{C}\colon|z|=1\}$ and note that
\[
g|_\Lambda\colon\Lambda\to\Lambda
\]
is topologically conjugate to the circle rotation
\[
R_{\tau^{-1}}\colon S^1\to S^1
\]
defined by $R_{\tau^{-1}}(z)=z\cdot e^{2\pi i\tau^{-1}}$ for all $z\in S^1$, i.e., there is a homeomorphism $h\colon\Lambda\to S^1$ such that
\[
h\circ g|_\Lambda=R_{\tau^{-1}}\circ h.
\]
If $\tau\in\mathbb{Q}$, then every $z\in S^1$ is a periodic point for $R_{\tau^{-1}}$ and so every $x\in\Lambda$ is a periodic point for $g|_\Lambda$. If $\tau\notin\mathbb{Q}$, then  $R_{\tau^{-1}}$ is {\em minimal} and so is $g|_\Lambda$. In both cases, $R_{\tau^{-1}}$ is {\em equicontinuous} and so is $g|_{\Lambda}$.

We prove the following claim.

\begin{claim}
There is a sequence $(\delta_n)_{n\ge2}$ of positive numbers such that
\[
{\rm DC1}_n^{\delta_n}(M,g)
\]
is a residual subset of $M^n$ for all $n\ge2$.
\end{claim}

\begin{rem}
\normalfont
As stated in Remark 1.1, it follows from Mycielski's theorem and the above claim that $g$ exhibits ${\rm DC1}_n$ for all $n\ge2$.
\end{rem}

We say that a map $f\colon X\to X$ is {\em equicontinuous} if for every $\epsilon>0$, there is $\delta>0$ such that $d(x,y)\le\delta$ implies
\[
\sup_{i\ge0}d(f^i(x),f^i(y))\le\epsilon
\]
for all $x,y\in X$. We know that if an equicontinuous map $f\colon X\to X$ is surjective, then $f$ is a homeomorphism and $f^{-1}$ is also equicontinuous (cf.\:\cite{AG, Ma}). If $f\colon X\to X$ is an equicontinuous homeomorphism, then $f$ is {\em distal}, i.e.,
\[
\inf_{i\ge0}\min_{1\le j<k\le n}d(f^i(x_j),f^i(x_k))>0
\]
for all $n\ge2$ and distinct $x_1,x_2,\dots,x_n\in X$.

We recall the basic definition of minimality.

\begin{defi}
\normalfont
For a continuous map $f\colon X\to X$, a subset $S$ of $X$ is said to be {\em $f$-invariant} if $f(S)\subset S$. A closed $f$-invariant subset $K$ of $X$ is called a {\em minimal set} for $f$ if closed $f$-invariant subsets of $K$ are only $\emptyset$ and $K$. This is equivalent to that $K=\overline{\{f^i(x)\colon i\ge0\}}$ for all $x\in K$. We say that a continuous map $f\colon X\to X$ is {\em minimal} if $X$ is a minimal set for $f$.
\end{defi}

\begin{rem}
\normalfont
Since a minimal continuous map $f\colon X\to X$ is surjective, every minimal equicontinuous map $f\colon X\to X$ is an equicontinuous homeomorphism. We know that if $f\colon X\to X$ is an equicontinuous homeomorphism, then $X$ is a disjoint union of minimal sets for $f$. We also know that every minimal equicontinuous homeomorphism $f\colon X\to X$ is topologically conjugate to a minimal rotation
\[
R_a\colon G\to G
\]
of a compact Abelian group $(G,+)$ where $a\in G$ and $R_a(z)=z+a$ for all $z\in G$ (see Theorem 2.42 of \cite{Ku}).  
\end{rem}

Let $f\colon X\to X$ be a continuous map. We say that a subset $S$ of $X$ is a {\em distal set} for $f$ if
\[
\inf_{i\ge0}\min_{1\le j<k\le n}d(f^i(x_j),f^i(x_k))>0
\]
for all $n\ge2$ and distinct $x_1,x_2,\dots,x_n\in S$. Note that a subset $S$ of $X$ with $|S|\le1$ is by definition a distal set for $f$. For $x\in X$ and $\epsilon>0$, we define a subset $V_\epsilon^s(x)$ of $X$ by
\[
V_\epsilon^s(x)=\{y\in X\colon\limsup_{i\to\infty}d(f^i(x),f^i(y))\le\epsilon\}.
\]

In order to prove the above claim, it is sufficient to show the following lemma.

\begin{lem}
Let $f\colon X\to X$ be a continuous map and let $S$ be a subset of $X$. If
\begin{itemize}
\item $S$ is a distal set for $f$,
\item $X=\overline{V_\epsilon^s(x)}$ for all $x\in S$ and $\epsilon>0$,
\end{itemize}
then for any $2\le n<|S|+1$, ${\rm DC1}_n^{\delta_n}(X,f)$ is a residual subset of $X^n$ for some $\delta_n>0$.
\end{lem}

\begin{proof}
Given $2\le n<|S|+1$, let
\begin{align*}
A_\delta&=\{(x_1,x_2,\dots,x_n)\in X^n\colon\\
&\qquad\qquad\qquad\limsup_{m\to\infty}\frac{1}{m}|\{0\le i\le m-1\colon\min_{1\le j<k\le n}d(f^i(x_j),f^i(x_k))>\delta\}|=1\}\\
&=\bigcap_{p\ge1}\bigcap_{q\ge1}\bigcup_{m\ge q}\{(x_1,x_2,\dots,x_n)\in X^n\colon\\
&\qquad\qquad\qquad\frac{1}{m}|\{0\le i\le m-1\colon\min_{1\le j<k\le n}d(f^i(x_j),f^i(x_k))>\delta\}|>1-\frac{1}{p}\},
\end{align*}
$\delta>0$, and let
\begin{align*}
B&=\bigcap_{\epsilon>0}\{(x_1,x_2,\dots,x_n)\in X^n\colon\\
&\qquad\qquad\qquad\qquad\limsup_{m\to\infty}\frac{1}{m}|\{0\le i\le m-1\colon\max_{1\le j<k\le n}d(f^i(x_j),f^i(x_k))<\epsilon\}|=1\}\\
&=\bigcap_{l\ge1}\bigcap_{p\ge1}\bigcap_{q\ge1}\bigcup_{m\ge q}\{(x_1,x_2,\dots,x_n)\in X^n\colon\\
&\qquad\qquad\qquad\qquad\frac{1}{m}|\{0\le i\le m-1\colon\max_{1\le j<k\le n}d(f^i(x_j),f^i(x_k))<\frac{1}{l}\}|>1-\frac{1}{p}\}.
\end{align*}
Note that
\[
{\rm DC1}_n^\delta(X,f)=A_\delta\cap B
\]
for all $\delta>0$. Since $S$ is a distal set for $f$, there are $(a_1,a_2,\dots,a_n)\in S^n$ and $\delta_n,\epsilon_n>0$ such that
\[
\inf_{i\ge0}\min_{1\le j<k\le n}d(f^i(a_j),f^i(a_k))>\delta_n+\epsilon_n.
\]
Since
\[
V_{\epsilon_n}^s(a_1)\times V_{\epsilon_n}^s(a_2)\times\cdots\times V_{\epsilon_n}^s(a_n)
\]
is a dense subset of $X^n$ and contained in
\begin{align*}
&\bigcup_{m\ge q}\{(x_1,x_2,\dots,x_n)\in X^n\colon\\
&\qquad\qquad\qquad\frac{1}{m}|\{0\le i\le m-1\colon\min_{1\le j<k\le n}d(f^i(x_j),f^i(x_k))>\delta_n\}|>1-\frac{1}{p}\}
\end{align*}
for all $p,q\ge1$, we see that $A_{\delta_n}$ is a residual subset of $X^n$. By taking $a\in S$, since $V_{\frac{1}{4l}}^s(a)^n$, $l\ge1$, is a dense subset of $X^n$ and contained in
\begin{align*}
&\bigcup_{m\ge q}\{(x_1,x_2,\dots,x_n)\in X^n\colon\\
&\qquad\qquad\qquad\frac{1}{m}|\{0\le i\le m-1\colon\max_{1\le j<k\le n}d(f^i(x_j),f^i(x_k))<\frac{1}{l}\}|>1-\frac{1}{p}\}
\end{align*}
for all $l,p,q\ge1$, we see that $B$ is a residual subset of $X^n$. It follows that ${\rm DC1}_n^{\delta_n}(X,f)$ is a residual subset of $X^n$, thus the proof has been completed. 
\end{proof}

\begin{rem}
\normalfont
We should note that, in the context of shadowing, a similar argument was presented by Li, Li, and Tu in Section 3 of \cite{LLT}  (see, in particular, Lemma 3.2; Lemma 3.3; and Theorem 3.4 of \cite{LLT}). 
\end{rem}

\begin{rem}
\normalfont
For a continuous map $f\colon X\to X$, let $Per(f)$ denote the set of periodic points for $f$:
\[
Per(f)=\{y\in X\colon f^i(y)=y\:\:\text{for some $i>0$}\}.
\]
For $x\in X$, the $\omega$-limit set $\omega(x,f)$ is defined as the set of $y\in X$ such that
\[
\lim_{j\to\infty}f^{i_j}(x)=y
\]
for some sequence $0\le i_1<i_2<\cdots$. Given any $x\in X$, $\{f^i(x)\colon i\ge0\}$ is a distal set for $f$ exactly if
\begin{itemize}
\item[--] $|\{f^i(x)\colon i\ge0\}|<\infty$ and $x\in Per(f)$; or
\item[--] $|\{f^i(x)\colon i\ge0\}|=\infty$ and $\omega(x,f)\cap Per(f)=\emptyset$.
\end{itemize}
If $f$ is minimal, then $\{f^i(x)\colon i\ge0\}$ is a distal set for $f$ for all $x\in X$. We also remark that if $f$ is surjective, then for any $x\in X$ and $\epsilon>0$, $X=\overline{V_\epsilon^s(x)}$ implies $X=\overline{V_\epsilon^s(f^i(x))}$ for all $i\ge0$.
\end{rem}

By this remark, we obtain the following corollary of Lemma 1.1.

\begin{cor}
Let $f\colon X\to X$ be a continuous map and let $\Lambda$ be a minimal set for $f$. If
\begin{itemize}
\item $f$ is surjective,
\item there is $x\in\Lambda$ such that $X=\overline{V_\epsilon^s(x)}$ for all $\epsilon>0$, 
\end{itemize}
then for any $2\le n<|\Lambda|+1$, ${\rm DC1}_n^{\delta_n}(X,f)$ is a residual subset of $X^n$ for some $\delta_n>0$.
\end{cor}

We have observed that the time-one map of a mixing Anosov flow exhibits ${\rm DC1}_n$ for all $n\ge2$. The proof presented above is based on the fact that the strong stable manifold of a periodic point is dense in the phase space. By relaxing this condition, we can derive sufficient conditions for ${\rm DC1}_n$, $n\ge2$.

The main result of this paper is the following theorem. A continuous map $f\colon X\to X$ is said to be {\em transitive} if for any non-empty open subsets $U,V$ of $X$, it holds that $f^i(U)\cap V\ne\emptyset$ for some $i\ge 0$.
\begin{thm}
Given a continuous map $f\colon X\to X$ and a closed $f$-invariant subset $\Lambda$ of $X$, if the following conditions are satisfied
\begin{itemize}
\item[(1)] $f|_\Lambda\colon\Lambda\to\Lambda$ is minimal and equicontinuous,
\item[(2)] $X=\overline{\bigcup_{x\in\Lambda}V_\epsilon^s(x)}$ for all $\epsilon>0$,
\item[(3)] $f|_\Lambda\times f\colon\Lambda\times X\to\Lambda\times X$ is transitive,
\end{itemize}
then for any $2\le n<|\Lambda|+1$, ${\rm DC1}_n^{\delta_n}(X,f)$ is a residual subset of $X^n$ for some $\delta_n>0$.
\end{thm}

\begin{rem}
\normalfont
In \cite{AGHSY}, it is shown that for a non-trivial transitive continuous map $f\colon X\to X$, if there is a closed $f$-invariant subset $\Lambda$ of $X$ such that
\[
f|_\Lambda\times f\colon\Lambda\times X\to\Lambda\times X
\]
is transitive, then $f$ exhibits (dense and uniform) Li--Yorke chaos (see Theorem 3.1 of \cite{AGHSY}).
\end{rem}

\begin{rem}
\normalfont
Let $f\colon X\to X,g\colon Y\to Y$ be continuous self-maps of compact metric spaces $X,Y$. It is known that $f\times g\colon X\times Y\to X\times Y$
is transitive if
\begin{itemize}
\item[(1)] $f$ is totally transitive; and $Y$ is a periodic orbit or an odometer,
\item[(2)] $f$ is weakly scattering; and $g$ is minimal and equicontinuous,
\item[(3)] $f$ is scattering; and $g$ is minimal
\end{itemize}
(see, e.g., \cite{AGHSY,HY} and also Remark 1.8 for details).
\end{rem}

We recall the notion of {\em Furstenberg families}.

\begin{defi}
\normalfont
Let $\mathbb{N}_0=\{0\}\cup\mathbb{N}=\{0,1,2,\dots\}$ and let $\mathcal{F}\subset2^{\mathbb{N}_0}$. We say that
$\mathcal{F}$ is a {\em Furstenberg family} if the following conditions are satisfied
\begin{itemize}
\item (hereditary upward) For any $A,B\subset\mathbb{N}_0$, $A\in\mathcal{F}$ and $A\subset B$ implies $B\in\mathcal{F}$,
\item (proper) $\mathcal{F}\ne\emptyset$ and $\mathcal{F}\ne2^{\mathbb{N}_0}$.
\end{itemize}
\end{defi}

For a Furstenberg family $\mathcal{F}$, we define its dual family $\mathcal{F}^\ast$ by
\[
\mathcal{F}^\ast=\{A\subset\mathbb{N}_0\colon A\cap B\ne\emptyset\:\:\text{for all $B\in\mathcal{F}$}\},
\]
which is also a Furstenberg family. For $A\subset\mathbb{N}_0$ with $0\in A$, let $\Delta(A)$ (or $A-A$) denote the set of $i\in\mathbb{N}_0$ such that $i=k-j$ for some $j,k\in A$. For a Furstenberg family $\mathcal{F}$, we define a Furstenberg family $\Delta(\mathcal{F})$ by
\[
\Delta(\mathcal{F})=\{B\subset\mathbb{N}_0\colon\text{$\Delta(A)\subset B$ for some $A\in\mathcal{F}$}\}.
\]

Let $f\colon X\to X$ be a continuous map and let $\mathcal{F}$ be a Furstenberg family. For $x\in X$ and $\epsilon>0$, let
\[
N(x,\epsilon)=\{i\in\mathbb{N}_0\colon d(x,f^i(x))\le\epsilon\}.
\]
We say that $x\in X$ is {\em $\mathcal{F}$-recurrent} if $N(x,\epsilon)\in\mathcal{F}$ for all $\epsilon>0$. We denote by $R(f,\mathcal{F})$ the set of $\mathcal{F}$-recurrent points for $f$. For subsets $A,B$ of $X$, let
\[
N(A,B)=\{i\in\mathbb{N}_0\colon f^i(A)\cap B\ne\emptyset\}.
\]
We say that $f$ is {\em $\mathcal{F}$-transitive} if $N(U,V)\in\mathcal{F}$ for all non-empty open subsets $U,V$ of $X$. 

We obtain the following corollary of Theorem 1.1.

\begin{cor}
Given a continuous map $f\colon X\to X$, a closed $f$-invariant subset $\Lambda$ of $X$, and a Furstenberg family $\mathcal{F}$, if the following conditions are satisfied
\begin{itemize}
\item[(1)] $f|_{\Lambda}\colon\Lambda\to\Lambda$ is minimal and equicontinuous,
\item[(2)] $X=\overline{\bigcup_{x\in\Lambda}V_\epsilon^s(x)}$ for all $\epsilon>0$,
\item[(3)] $\Lambda\cap R(f,\mathcal{F})\ne\emptyset$,
\item[(4)] $f$ is $\Delta(\mathcal{F})^\ast$-transitive,
\end{itemize}
then for any $2\le n<|\Lambda|+1$, ${\rm DC1}_n^{\delta_n}(X,f)$ is a residual subset of $X^n$ for some $\delta_n>0$.
\end{cor}

In fact, this corollary is a direct consequence of Theorem 1.1 and the following two lemmas.

\begin{lem}
Let $f\colon X\to X$ be a continuous map and let $\Lambda$ be a closed $f$-invariant subset of $X$. If $f|_{\Lambda}\colon\Lambda\to\Lambda$ is minimal and equicontinuous, then for any Furstenberg family $\mathcal{F}$, $\Lambda\cap R(f,\mathcal{F})\ne\emptyset$ implies $\Lambda\subset R(f,\mathcal{F})$. 
\end{lem}

\begin{lem}
Let $f\colon X\to X,g\colon Y\to Y$ be continuous self-maps of compact metric spaces $X,Y$. For a Furstenberg family $\mathcal{F}$, if
\begin{itemize}
\item $f$ is $\Delta(\mathcal{F})^\ast$-transitive,
\item $g$ is transitive and satisfies $Y=R(g,\mathcal{F})$,
\end{itemize}
then $f\times g\colon X\times Y\to X\times Y$ is transitive.
\end{lem}

In Section 2, we prove these lemmas. We should note that Lemma 1.3 can be prove by a similar argument as in the proof of Theorem 4.9 in \cite{HY}, however we prove it for the sake of completeness.

\begin{rem}
\normalfont
Let $f\colon X\to X$ be a continuous map and let $\Lambda$ be a closed $f$-invariant subset of $X$.
\begin{itemize}
\item[(1)] For $k\ge1$, let $k\mathbb{N}_0=\{0,k,2k,3k,\dots\}$. We define a Furstenberg family $\mathcal{F}_{\rm rr}$ by
\[
\mathcal{F}_{\rm rr}=\{A\subset\mathbb{N}_0\colon\text{$k\mathbb{N}_0\subset A$ for some $k\ge1$}\}.
\]
Every $x\in R(f,\mathcal{F}_{\rm rr})$ is called a {\em regularly recurrent} point for $f$. If $f|_{\Lambda}$ is minimal and equicontinuous, then by Lemma 1.2, $\Lambda\cap R(f,\mathcal{F}_{\rm rr})\ne\emptyset$ implies $\Lambda\subset R(f,\mathcal{F}_{rr})$. We know that whenever $f|_{\Lambda}$ is minimal, $\Lambda\subset R(f,\mathcal{F}_{\rm rr})$ holds exactly if $\Lambda$ is a periodic orbit for $f$ (if $|\Lambda|<\infty$) or an odometer (if $|\Lambda|=\infty$) (see Corollary 2.5 of \cite{BK} in which an odometer is called an adding machine). In both cases, $f|_{\Lambda}$ is equicontinuous. It is easy to see that $f$ is $\mathcal{F}_{\rm rr}^\ast$-transitive if and only if $f$ is {\em totally transitive}, i.e., $f^k$ is transitive for all $k\ge1$. We also see that $f$ is totally transitive if and only if
\[
f\times g\colon X\times Y\to X\times Y
\]
is transitive for every continuous self-map $g\colon Y\to Y$ such that $Y$ is a periodic orbit or an odometer.
\item[(2)] We define a Furstenberg family $\mathcal{F}_{\rm b}$ as the set of $A\subset\mathbb{N}_0$ such that there are
\begin{itemize}
\item a minimal equicontinuous self-map $g\colon Y\to Y$ of a compact metric space $Y$,
\item $y\in Y$ and $\delta>0$
\end{itemize}
such that $N(y,\delta)\subset A$. If $f|_{\Lambda}$ is minimal and equicontinuous, then $\Lambda\subset R(f,\mathcal{F}_{\rm b})$. We know that $f$ is $\Delta(\mathcal{F}_{\rm b})^\ast$-transitive if and only if $f$ is {\em weakly scattering}, i.e.,
\[
f\times g\colon X\times Y\to X\times Y
\]
is transitive for every minimal equicontinuous self-map $g\colon Y\to Y$ of a compact metric space $Y$ (see Theorem 4.12 of \cite{HY}).
\item[(3)] We define two Furstenberg families $\mathcal{F}_{\rm s}$ and $\mathcal{F}_{\rm t}$ by
\[
\mathcal{F}_{\rm s}=\{A\subset\mathbb{N}_0\colon\text{$\exists k\ge1$ s.t.\:$A\cap\{i,i+1,\dots,i+k-1\}\ne\emptyset$ for $\forall i\ge0$}\}
\]
and
\[
\mathcal{F}_{\rm t}=\{B\subset\mathbb{N}_0\colon\text{for $\forall j\ge1$ $\exists i_j\ge0$ s.t.\:$\{i_j,i_j+1,\dots,i_j+j-1\}\subset B$}\}.
\]
Note that $\mathcal{F}_{\rm s}^\ast=\mathcal{F}_{\rm t}$. It is well-known that if $f|_{\Lambda}$ is minimal, then $\Lambda\subset R(f,\mathcal{F}_{\rm s})$. We know that
\begin{itemize}
\item $f$ is $\Delta(\mathcal{F}_{\rm s})^\ast$-transitive if and only if $f$ is {\em scattering}, i.e.,
\[
f\times g\colon X\times Y\to X\times Y
\]
is transitive for every minimal continuous self-map $g\colon Y\to Y$ of a compact metric space $Y$ (see Theorem 4.10 of \cite{HY}),
\item $f$ is $\mathcal{F}_{\rm t}$-transitive if and only if $f$ is {\em weakly mixing}, i.e.,
\[
f\times f\colon X\times X\to X\times X
\]
is transitive (see Proposition 7.2 of \cite{A}).
\end{itemize}
\item[(4)] Since
\[
\mathcal{F}_{\rm rr}\subset\Delta(\mathcal{F}_{\rm b})\subset\Delta(\mathcal{F}_{\rm s})\subset\mathcal{F}_{\rm s},
\]
we have
\[
\mathcal{F}_{\rm t}\subset\Delta(\mathcal{F}_{\rm s})^\ast\subset\Delta(\mathcal{F}_{\rm b})^\ast\subset\mathcal{F}_{rr}^\ast,
\]
which implies that for any continuous map $f\colon X\to X$, $f$ is
\[
\text{weakly mixing $\implies$ scattering $\implies$ weakly scattering $\implies$ totally transitive}.
\]
\end{itemize}
\end{rem}

This paper consists of two sections. In the next section, we prove Theorem 1.1, Lemmas 1.2 and 1.3.
 
\section{Proofs of Theorems 1.1, Lemmas 1.2 and 1.3}

In this section, we prove Theorems 1.1, Lemmas 1.2 and 1.3. The proof of Theorem 1.1 is structured as a step-by-step proof of a series of lemmas, the meaning of each of which should be clear.

Let $f\colon X\to X$ be a continuous map and let $\Lambda$ be a closed $f$-invariant subset of $X$ such that $f|_{\Lambda}\colon\Lambda\to\Lambda$ is minimal and equicontinuous.

\begin{lem}
For every $\epsilon>0$, $X=\overline{\bigcup_{x\in\Lambda}V_\epsilon^s(x)}$ implies $X=\bigcup_{x\in\Lambda}\overline{V_{2\epsilon}^s(x)}$.
\end{lem}

\begin{proof}
If $X=\overline{\bigcup_{x\in\Lambda}V_\epsilon^s(x)}$, then for every $y\in X$, there are sequences $x_j\in\Lambda$, $y_j\in V_\epsilon^s(x_j)$, $j\ge1$, and  $x\in\Lambda$ such that $\lim_{j\to\infty}x_j=x$ and $\lim_{j\to\infty}y_j=y$. Since $f|_{\Lambda}\colon\Lambda\to\Lambda$ is equicontinuous, we have $\delta>0$ such that $d(x,z)\le\delta$ implies
\[
\sup_{i\ge0}d(f^i(x),f^i(z))\le\epsilon
\]
and so $V_\epsilon^s(z)\subset V_{2\epsilon}^s(x)$ for all $z\in\Lambda$. It follows that $y_j\in V_\epsilon^s(x_j)\subset V_{2\epsilon}^s(x)$ for all sufficiently large $j\ge1$, which implies $y\in\overline{V_{2\epsilon}^s(x)}$. Since $y\in X$ is arbitrary, we obtain  $X=\bigcup_{x\in\Lambda}\overline{V_{2\epsilon}^s(x)}$, proving the lemma.
\end{proof}

\begin{lem}
Let $\Gamma$ be a countable dense subset of $\Lambda$. For every $\epsilon>0$, $X=\bigcup_{x\in\Lambda}\overline{V_\epsilon^s(x)}$ implies $X=\bigcup_{x\in\Gamma}\overline{V_{2\epsilon}^s(x)}$.
\end{lem}

\begin{proof}
Let $\epsilon>0$. Given any $x\in\Lambda$, since $f|_{\Lambda}\colon\Lambda\to\Lambda$ is equicontinuous, there is $\delta>0$ such that $d(x,z)\le\delta$ implies
\[
\sup_{i\ge0}d(f^i(x),f^i(z))\le\epsilon
\]
and so $V_\epsilon^s(x)\subset V_{2\epsilon}^s(z)$ for all $z\in\Lambda$. Since $\Gamma$ is dense in $\Lambda$, by taking $z\in\Gamma$ with $d(x,z)\le\delta$, we obtain $V_\epsilon^s(x)\subset V_{2\epsilon}^s(z)$ and so
\[
\overline{V_\epsilon^s(x)}\subset\overline{V_{2\epsilon}^s(z)}.
\]
Since $x\in\Lambda$ is arbitrary, it follows that
\[
\bigcup_{x\in\Lambda}\overline{V_\epsilon^s(x)}\subset\bigcup_{z\in\Gamma}\overline{V_{2\epsilon}^s(z)}.
\] 
If $X=\bigcup_{x\in\Lambda}\overline{V_\epsilon^s(x)}$, then we obtain
\[
X=\bigcup_{x\in\Lambda}\overline{V_\epsilon^s(x)}=\bigcup_{z\in\Gamma}\overline{V_{2\epsilon}^s(z)},
\]
completing the proof.
\end{proof}

\begin{lem}
For any $x\in\Lambda$ and $\epsilon>0$, if ${\rm int}[\overline{V_\epsilon^s(x)}]\ne\emptyset$, and if
\[
f|_\Lambda\times f\colon\Lambda\times X\to\Lambda\times X
\]
is transitive, then $X=\overline{V_{2\epsilon}^s(x)}$.
\end{lem}

\begin{proof}
Let $V$ be a non-empty open subset of $X$. Since $x\in\Lambda$ and $f|_\Lambda$ is equicontinuous, there is $\delta>0$ such that $d(x,b)\le\delta$ implies
\[
\sup_{i\ge0}d(f^i(x),f^i(b))\le\epsilon
\]
and so $V_\epsilon^s(b)\subset V_{2\epsilon}^s(x)$ for all $b\in\Lambda$. Again since $x\in\Lambda$ and $f|_\Lambda$ is equicontinuous, there is $\gamma>0$ such that $d(x,a)\le\gamma$ implies
\[
\sup_{i\ge0}d(f^i(x),f^i(a))\le\delta/2
\]
for all $a\in\Lambda$. Since ${\rm int}[\overline{V_\epsilon^s(x)}]\ne\emptyset$ and $f|_\Lambda\times f$ is transitive, there are $z\in\Lambda$, $p\in X$, and $i\ge0$ such that
\begin{itemize}
\item $d(x,z)\le\gamma$ and $d(x,f^i(z))\le\delta/2$,
\item $p\in{\rm int}[\overline{V_\epsilon^s(x)}]$ and $f^i(p)\in V$.
\end{itemize}
By $f^i(x)\in\Lambda$ and
\[
d(x,f^i(x))\le d(x,f^i(z))+d(f^i(x),f^i(z))\le\delta/2+\delta/2=\delta,
\]
we obtain $V_\epsilon^s(f^i(x))\subset V^s_{2\epsilon}(x)$. Since
\[
p\in{\rm int}[\overline{V_\epsilon^s(x)}]\subset\overline{V_\epsilon^s(x)}
\]
we have $f^i(q)\in V$ for some $q\in V_\epsilon^s(x)$. It follows that
\[
f^i(q)\in f^i(V_\epsilon^s(x))\subset V_\epsilon^s(f^i(x))\subset V_{2\epsilon}^s(x)
\]
and so
\[
f^i(q)\in V\cap V_{2\epsilon}^s(x).
\]
Since $V$ is arbitrary, we obtain $X=\overline{V_{2\epsilon}^s(x)}$, proving the lemma.
\end{proof}
\begin{lem}
Let $x,y\in\Lambda$ and $\epsilon>0$. If $X=\overline{V_\epsilon^s(x)}$, and if $f$ is surjective, then $X=\overline{V_{2\epsilon}^s(y)}$.
\end{lem}

\begin{proof}
Let $V$ be a non-empty open subset of $X$. Since $y\in\Lambda$ and $f|_{\Lambda}$ is equicontinuous, there is $\delta>0$ such that $d(y,z)\le\delta$ implies
\[
\sup_{i\ge0}d(f^i(y),f^i(z))\le\epsilon
\]
and so $V_\epsilon^s(z)\subset V_{2\epsilon}^s(y)$ for all $z\in\Lambda$. Since $x\in\Lambda$ and $f|_{\Lambda}$ is minimal, we have
\[
\Lambda=\overline{\{f^i(x)\colon i\ge0\}},
\]
which implies $d(y,f^i(x))\le\delta$ and so $V_\epsilon^s(f^i(x))\subset V_{2\epsilon}^s(y)$ for some $i\ge0$. Since $f$ is surjective, we have $f^i(p)\in V$ for some $p\in X$. By $X=\overline{V_\epsilon^s(x)}$, we obtain $f^i(q)\in V$ for some $q\in V_\epsilon^s(x)$. It follows that
\[
f^i(q)\in f^i(V_\epsilon^s(x))\subset V_\epsilon^s(f^i(x))\subset V_{2\epsilon}^s(y);
\]
therefore,
\[
f^i(q)\in V\cap V_{2\epsilon}^s(y).
\]
Since $V$ is arbitrary, we obtain $X=\overline{V_{2\epsilon}^s(y)}$, completing the proof.
\end{proof}

\begin{lem}
If $X=\overline{\bigcup_{x\in\Lambda}V_\epsilon^s(x)}$ for all $\epsilon>0$, and if
\[
f|_\Lambda\times f\colon\Lambda\times X\to\Lambda\times X
\]
is transitive, then $X=\overline{V_\epsilon^s(x)}$ for all $x\in\Lambda$ and $\epsilon>0$.
\end{lem}

\begin{proof}
Given any $\epsilon>0$, by Lemma 2.1, we have $X=\bigcup_{x\in\Lambda}\overline{V_{2\epsilon}^s(x)}$. By Lemma 2.2, taking a countable dense subset $\Gamma$ of $\Lambda$, we have $X=\bigcup_{x\in\Gamma}\overline{V_{4\epsilon}^s(x)}$. By Baire category theorem, we obtain
\[
{\rm int}[\overline{V_{4\epsilon}^s(p)}]\ne\emptyset
\]
for some $p\in\Gamma$. Since
\[
f|_\Lambda\times f\colon\Lambda\times X\to\Lambda\times X
\]
is transitive, Lemma 2.3 implies $X=\overline{V_{8\epsilon}^s(p)}$. Note that $f$ is surjective because $f|_\Lambda\times f$ and so $f$ is transitive. By Lemma 2.4, we obtain $X=\overline{V_{16\epsilon}^s(x)}$ for all $x\in\Lambda$. Since $\epsilon>0$ is arbitrary, we obtain $X=\overline{V_\epsilon^s(x)}$ for all $x\in\Lambda$ and $\epsilon>0$, thus the lemma has been proved.
\end{proof}
 
Let us prove Theorem 1.1.

\begin{proof}[Proof of Theorem 1.1]
If a continuous map $f\colon X\to X$ and a closed $f$-invariant subset $\Lambda$ of $X$ satisfy conditions (1)--(3) in Theorem 1.1, then by Lemma 2.5, we have $X=\overline{V_\epsilon^s(x)}$ for all $x\in\Lambda$ and $\epsilon>0$. From Lemma 1.1, it follows that for any $2\le n<|\Lambda|+1$, ${\rm DC1}_n^{\delta_n}(X,f)$ is a residual subset of $X^n$ for some $\delta_n>0$. This completes the proof of Theorem 1.1.
\end{proof}

Finally, we prove Lemmas 1.2 and 1.3.

\begin{proof}[Proof of Lemma 1.2]
Given any $q\in\Lambda$ and $\epsilon>0$, since $f|_{\Lambda}\colon\Lambda\to\Lambda$ is equicontinuous, there is $\delta>0$ such that $d(q,y)\le\delta$ implies
\[
\sup_{j\ge0}d(f^j(q),f^j(y))\le\epsilon/3
\]
for all $y\in\Lambda$. Note that $f(x)\in R(f,\mathcal{F})$ for all $x\in R(f,\mathcal{F})$. By taking $p\in\Lambda\cap R(f,\mathcal{F})$, we obtain $f^i(p)\in R(f,\mathcal{F})$ for all $i\ge0$. Since $f|_{\Lambda}\colon\Lambda\to\Lambda$ is minimal, $p$ satisfies
\[
\Lambda=\overline{\{f^i(p)\colon i\ge0\}}
\]
and so $d(q,f^i(p))\le\delta$ for some $i\ge0$. By $f^i(p)\in R(f,\mathcal{F})$, we obtain
\[
N(f^i(p),\epsilon/3)\in\mathcal{F}.
\]
Since $f^i(p)\in\Lambda$ and $d(q,f^i(p))\le\delta$, we have
\[
d(q,f^j(q))\le d(q,f^i(p))+d(f^i(p),f^j(f^i(p)))+d(f^j(f^i(p)),f^j(q))\le\epsilon/3+\epsilon/3+\epsilon/3=\epsilon
\]
for all $j\in N(f^i(p),\epsilon/3)$. It follows that
\[
N(f^i(p),\epsilon/3)\subset N(q,\epsilon)
\]
and so
\[
N(q,\epsilon)\in\mathcal{F}.
\]
Since $q\in\Lambda$ and $\epsilon>0$ are arbitrary, we obtain $\Lambda\subset R(f,\mathcal{F})$, completing the proof.
\end{proof}

\begin{proof}[Proof of Lemma 1.3]
For $A\subset\mathbb{N}_0$ and $j\ge0$, let $A+j=\{i+j\colon i\in A\}$. Let $U_1,V_1$ be  non-empty open subsets of $Y$. Since $g$ is transitive, there are $y\in U_1$ and $j\ge0$ such that $g^j(y)\in V_1$. We take $\delta>0$ such that $B_\delta(y)\in U_1$ and $g^j(B_\delta(y))\subset V_1$ where $B_\delta(y)$ is the closed $\delta$-ball centered at $y$. It follows that
\[
N(B_\delta(y),B_\delta(y))+j\subset N(U_1,V_1)
\]
because for every $i\in N(B_\delta(y),B_\delta(y))$, there is $p\in B_\delta(y)$ such that $g^i(p)\in B_\delta(y)$ which implies $p\in U_1$ and $g^{i+j}(p)=g^j(g^i(p))\in g^j(B_\delta(y))\subset V_1$; therefore, $i+j\in N(U_1,V_1)$. Note that $0\in N(y,\delta)$. We have
\[
\Delta(N(y,\delta))\subset N(B_\delta(y),B_\delta(y))
\]
because it holds for any $k,l\in N(y,\delta)$ with $k\le l$, $g^k(y)\in B_\delta(y)$ and $g^{l-k}(g^k(y))=g^l(y)\in B_\delta(y)$; therefore, $l-k\in N(B_\delta(y),B_\delta(y))$. Let $U,V$ be  non-empty open subsets of $X$. Since $y\in Y$ and $Y=R(g,\mathcal{F})$, we have $N(y,\delta)\in\mathcal{F}$ and so $\Delta(N(y,\delta))\in\Delta(\mathcal{F})$. Since $f$ is $\Delta(\mathcal{F})^\ast$-transitive, we obtain
\[
N(U,f^{-j}(V))\cap \Delta(N(y,\delta))\ne\emptyset
\]
and so
\[
N(U,f^{-j}(V))\cap N(B_\delta(y),B_\delta(y))\ne\emptyset,
\]
implying
\[
[N(U,f^{-j}(V))+j]\cap[N(B_\delta(y),B_\delta(y))+j]\ne\emptyset
\]
and so 
\[
[N(U,f^{-j}(V))+j]\cap N(U_1,V_1)\ne\emptyset.
\]
We have
\[
N(U,f^{-j}(V))+j\subset N(U,V)
\]
because for every $i\in N(U,f^{-j}(V))$, there is $x\in U$ such that $f^i(x)\in f^{-j}(V)$ and so $f^{i+j}(x)=f^j(f^i(x))\in V$; therefore $i+j\in N(U,V)$. It follows that
\[
N(U,V)\cap N(U_1,V_1)\ne\emptyset.
\]
Since $U_1,V_1,U,V$ are arbitrary, we conclude that
\[
f\times g\colon X\times Y\to X\times Y
\]
is transitive, proving the lemma.
\end{proof}

\end{document}